\documentclass[11pt, reqno]{amsart}
\usepackage{amsmath,amsopn,amssymb,amsthm,multicol}
\usepackage[cp1251]{inputenc}
\usepackage[english]{babel}
\usepackage[breaklinks=true,colorlinks=true,linkcolor=blue,citecolor=blue,urlcolor=blue]{hyperref}
\usepackage[dvips]{graphicx}

\renewenvironment{proof}[1][Proof]{\textbf{#1.} }
{\ \rule{0.5em}{0.5em}}

\DeclareMathOperator{\ad}{ad}
\DeclareMathOperator{\Ad}{Ad}
\DeclareMathOperator{\diag}{diag}
\DeclareMathOperator{\Ric}{Ric}

\renewcommand{\arraystretch}{1.5}

\newtheorem{theorem}{Theorem}
\newtheorem{pred}{Proposition}
\newtheorem{lem}{Lemma}

\newtheorem{remark}{Remark}
\newtheorem{ques}{Question}

\begin{document}

\title{Invariant Einstein metrics on Ledger~--~Obata spaces}

\author{Zhiqi~Chen}
\address{Zhiqi~Chen\newline
School of Mathematical Sciences and LPMC, Nankai University, \newline
Tianjin 300071, China}
\email{chenzhiqi@nankai.edu.cn}

\author{Yuri\u{i}~Nikonorov}
\address{Yu.\,G. Nikonorov \newline
Southern Mathematical Institute of Vladikavkaz Scientific Centre \newline
of the Russian Academy of Sciences, Vladikavkaz, Markus st. 22, \newline
362027, Russia}
\email{nikonorov2006@mail.ru}

\author{Yulia~Nikonorova}
\address{Yu.\,V. Nikonorova \newline
VITI MEPhI, Volgodonsk, Lenin st. 73/94, \newline
347360, Russia}
\email{nikonorova2009@mail.ru}

\begin{abstract}
In this paper, we study invariant Einstein metrics on Ledger~--~Obata spaces $F^m/\diag(F)$. In particular, we classify invariant Einstein metrics on
$F^4/\diag(F)$ and estimate the number of invariant Einstein metrics on Ledger~--~Obata spaces $F^{m}/\diag(F)$.
\vspace{2mm}


\vspace{2mm} \noindent Key words and phrases: Ledger~--~Obata space, Einstein metric, Riemannian manifold, compact Lie algebra.

\end{abstract}

\maketitle

\section{Introduction}

The spaces $F^m/\diag(F)$ are called Ledger~--~Obata spaces, where $F$ is a connected compact simple Lie group, $F^m=F\times F\times\cdots\times F$
($m$ factors and $m\geq 2$), and $\diag(F)=\{(X,X,\dots,X)|X\in F\}$. Ledger~--~Obata spaces were first introduced in \cite{LO1968} as a natural generalization
of symmetric spaces, since $F^2/\diag(F)$ is an irreducible symmetric space. Our main interest is to study invariant Einstein metrics on Ledger~--~Obata spaces.
\smallskip

A Riemannian metric is Einstein if the Ricci curvature is a constant multiple of the metric.
Various results on Einstein manifolds could be found in the book \cite{Bes} of A.L.~Besse and in more recent surveys \cite{NRS, Wang1, Wang2}.
The most interesting feature of invariant Einstein metrics on Ledger~--~Obata spaces $F^m/\diag(F)$ is the fact that
the number of these metrics does not depend on the structure of the simple Lie group $F$. Moreover, under a suitable parameterization
the search of such metrics on every Ledger~--~Obata space $F^m/\diag(F)$ is reduced to search the critical points of a special
real-valued function. The above results were given in \cite{Nik2002}, together with the result that $F^3/\diag(F)$
admits exactly two invariant Einstein metrics up to isometry and homothety.
\smallskip

On the other hand, the approach in \cite{Nik2002} is too complicated to classify invariant Einstein metrics even for $F^4/\diag(F)$.
In this paper, we develop another approach to classify invariant Einstein metrics on $F^4/\diag(F)$ and reprove main results of \cite{Nik2002} in a simpler way.
Note also that $F^4/\diag(F)$ (with special choices of Riemannian metrics)
could be considered as a {\it generalized Wallach space} (or a {\it three locally symmetric space}), see \cite{CKL, Nik2016, NRS}.
Hence, our results complete the classification of invariant Einstein metrics on generalized Wallach spaces \cite{CN}.
\smallskip

Recall that a Riemannian metric on the Ledger~--~Obata space $F^m/\diag(F)$ is called {\it standard} or {\it Killing} if it is induced by
the minus Killing form of the Lie algebra of $F^m$. The standard Rimennian metric $\rho_{st}$ on every Ledger~--~Obata space is Einstein (see \cite{Nik2002} or Example 2 in \cite{WZ1991}).
A~direct calculation shows that  $\Ric(\rho_{st})=\frac{m+2}{4m} \rho_{st}$ (for details see \cite{Nik2002}). By Theorem 5.2 of Chapter~X of~\cite{KN}, we know that the Riemannian manifold $(F^m/\diag(F),\rho_{st})$ is irreducible.
\smallskip

Let $\mathfrak{f}$ be the Lie algebra of $F$ and $B$ the Killing form of $\mathfrak{f}$. Then $\langle \cdot,\cdot \rangle =-B(\cdot,\cdot)$
is a bi-invariant inner product on $\mathfrak{f}$. Denote $p =\dim(\mathfrak{f})=\dim(F)$. We will also use the same notations $B$ and
$\langle \cdot,\cdot \rangle$ for the Killing form and the minus Killing form on the Lie algebra
$$
\mathfrak{g}:=n\mathfrak{f}=\underbrace{\mathfrak{f}\oplus \mathfrak{f} \oplus \cdots \oplus \mathfrak{f}}_{n\,\, \mbox{\tiny summands}}
$$
of the Lie group $G:=F^n$, and $\mathfrak{h}:=\diag(\mathfrak{f})=\{(X,X,\dots,X)\,|\, X\in \mathfrak{f}\}\subset \mathfrak{g}$
for the Lie algebra of $H:=\diag(F)\subset G=F^n$.
\smallskip

The basis of our approach in this paper is the observation
that the classification of invariant Einstein metrics on every Ledger~--~Obata space $F^{n+1}/\diag(F)$ is equivalent to the classification of $\Ad(H)$-invariant Einstein metrics on the Lie group $G$, i.e. the classification of $\ad ({\mathfrak h})$-invariant Einstein metrics on $\mathfrak g$.
The latter metrics were considered in the paper \cite{Nik1998}. It is well-known that the standard metric
(generated by the minus Killing form of $\mathfrak{g}$)
on $G$ is Einstein.
Note also that G.R.~Jensen \cite{Jen}
indicated one more Einstein metric on $F\times F$ which is neither isometric
nor homothetic to the standard metric but is $\Ad(H)$-invariant. In fact, this
metric is isometric (up to a positive multiple) to the standard metric on $F^3/\diag(F)$.
\smallskip

Let $M$ denote the set of $\ad(\mathfrak h)$-invariant inner
products on $\mathfrak g$ with volume 1 corresponding to
$\langle\cdot,\cdot\rangle$. For any $(\cdot,\cdot)\in M$, denote
by $S\bigl((\cdot,\cdot)\bigr)$ the scalar curvature $S$ of the
$\Ad(H)$-invariant metric on $G$ determined by $(\cdot,\cdot)$ on
$\mathfrak g$. By the result in~\cite{Nik1998}, $(\cdot,\cdot)$ is
Einstein if and only if $(\cdot,\cdot)$ is the critical point of
the scalar curvature function $S: M\rightarrow {\mathbb R}$. It is
a partial case for the variational principle for Einstein
Riemannian metrics on compact manifolds \cite{Bes}. In this paper,
we firstly find all invariant Einstein metrics on $F^4/\diag(F)$
by this way. Then we prove:

\begin{theorem}\label{thm1}
The Ledger~--~Obata space $F^4/\diag(F)$ admits exactly three invariant Einstein metrics up to isometry and homothety.
\end{theorem}

It is proved in \cite{Nik2002} that there are at least two invariant Einstein metrics on $F^{n+1}/\diag(F)$ for $n\geq 3$. Based on standard Einstein metrics on $F^{n+1}/\diag(F)$, we have the following theorem.

\begin{theorem}\label{thm2}
Every Ledger~--~Obata space $F^{n+1}/\diag(F)$ admits at least $p(n)$ invariant Einstein metrics up to isometry and homothety, where $p(n)$ is the number
of integer partitions of $n$. In particular, there are more than $\frac{1}{13n}\exp\left({\frac{5}{2}}\sqrt{n}\right)$ invariant Einstein metrics
for all $n$ up to isometry and homothety.
\end{theorem}

The paper is organized as follows. In section 2, we list all required constructions and results for Theorems~\ref{thm1} and~\ref{thm2}.
For example, the equivalence between the classification of invariant Einstein metrics on every Ledger~--~Obata space $F^{n+1}/\diag(F)$ and the classification of $\Ad(H)$-invariant Einstein metrics on the Lie group $F^n$, a new formula for the scalar curvature function, and sufficient conditions for isometric metrics. In section 3, we classify invariant Einstein metrics on $F^3/\diag(F)$ by our approach, which is simpler than that in \cite{Nik2002}. Section 4 is to compute all invariant
Einstein metrics on $F^4/\diag(F)$ by the new approach. Then we prove Theorem~\ref{thm1} in section 5. In section 6, we discuss the number of invariant Einstein metrics
on Ledger~--~Obata spaces $F^{n+1}/\diag(F)$ and prove Theorem~\ref{thm2}.
Moreover, we discuss the number of the critical points of the scalar curvature function for every Ledger~--~Obata space $F^{n+1}/\diag(F)$
and show in particular that the number is more than $ \bigl(1+\sqrt{2}\bigr)^{n-1}$. This estimate shows the technical difficulties in the problem of  classifying invariant Einstein metrics on Ledger~--~Obata spaces $F^{n+1}/\diag(F)$ for big values of $n$. Finally, we propose some unsolved questions.

\section{The set of metrics and the scalar curvature function}

In this section we give all required constructions and results for the proof of main results.

\subsection{Identification of metrics}
Note that the group $\widetilde{G}=F^{n+1}$ acts transitively on the Lie group $G=F^n$ by
$$
(a_1,a_2,\dots,a_n,a_{n+1})\circ (x_1,x_2,\dots, x_n)=(a_1x_1a_{n+1}^{-1},\,a_2x_2a_{n+1}^{-1},\,\dots,\, a_nx_na_{n+1}^{-1})
$$
with the isotropy group $\widetilde{H}=\diag(F)\subset \widetilde{G}=F^{n+1}$ at the unit of $G=F^n$,
which identifies $F^n$ with the Ledger~--~Obata space $F^{n+1}/\diag(F)$.
\smallskip

Since the adjoint representation of every real compact simple Lie algebra is real and irreducible, we easily get the following lemma.

\begin{lem}\label{simplealg}
Let $\mathfrak{f}$ be a real compact simple Lie algebra. Then every linear mapping $Q:\mathfrak{f} \rightarrow \mathfrak{f}$ commuting with all operators of the adjoint action $\ad(X):\mathfrak{f}\rightarrow \mathfrak{f}$, where $X\in \mathfrak{f}$, is proportional to the identity map.
\end{lem}

It should be noted that
there are several $\Ad(\widetilde{H})$-invariant complements to
$\widetilde{\frak{h}}$ in $\widetilde{\frak{g}}$. In particular,

\begin{lem}[Lemma 1.1 in \cite{Nik2002}]\label{strinmod}
Every irreducible $\Ad(\widetilde{H})$-invariant submodule in $\widetilde{\frak{g}}$ has the form
\begin{equation}\label{irred}
\left\{(\alpha_1 X, \alpha_2 X,\dots, \alpha_{n+1} X)\subset\widetilde{\frak{g}}\,|\, X \in \frak{f}\right\}
\end{equation}
for some fixed $\alpha_i$.
\end{lem}

For the reader's convenience we reproduce the draft of the proof of this lemma here. Let $\mathfrak{m}$ be an
irreducible $\Ad(\widetilde{H})$-invariant submodule in $\widetilde{\frak{g}}$.
Consider the linear maps $\theta_i:\mathfrak{m} \rightarrow \mathfrak{f}$, $i=1,2,\dots,n+1$, defined by $\theta_i(X)=X_i$, where $X=(X_1,X_2,\dots,X_n,X_{n+1})\in \mathfrak{m}\subset(n+1)\mathfrak{f}$. Since $\mathfrak{m}$ is $\ad(\widetilde{\frak{h}})$-invariant, for any $\widetilde{Z}=(Z,Z,\cdots,Z)\in \widetilde{\frak{h}}$ where $Z\in \frak{f}$, we see that $\theta_i([\widetilde{Z},X])=[Z,\theta_i(X)]$.
In particular, the image of $\theta_i$ is an ideal in $\frak{f}$, hence it is either $0$ or $\frak{f}$ itself.
At least one of such image should be non-zero. Without loss of generality, we may assume that this is the case for $i=n+1$. Hence, we may define the maps
$\kappa_i: \mathfrak{f} \rightarrow \mathfrak{f}$ by $\kappa_i=\theta_i\circ \theta_{n+1}^{-1}$ for any $i=1,2,\dots,n$.
By the above discussion we see that $\kappa_i$ is a linear map
commuting with all operators of the adjoint action
$\ad(Z):\mathfrak{f}\mapsto \mathfrak{f}$ for any $Z\in \mathfrak{f}$. By Lemma \ref{simplealg}, $\kappa_i$ is proportional to the identity map, which implies Lemma \ref{strinmod}.
\smallskip

Clearly, the isotropy subalgebra $\widetilde{\frak{h}}$ is a
submodule as in Lemma \ref{strinmod} with
$\alpha_1=\alpha_2=\cdots=\alpha_{n+1}\neq 0$. Denote by
$\frak{p}$ the union of all submodules \eqref{irred} with
$\alpha_1+\alpha_2+\cdots+\alpha_{n+1}=0$ (it could be represented
also as a direct sum of $n$ suitable submodules of this type).
Note that $\frak{p}$ is the $\langle \cdot,\cdot
\rangle$-orthogonal complement to $\widetilde{\frak{h}}$ in
$\widetilde{\frak{g}}$, where  $\langle \cdot,\cdot \rangle$ is
the minus Killing form of $\widetilde{\frak{g}}$.
\smallskip

Define the linear map $\varphi_{n+1}:\frak{p}\rightarrow \frak{g}=n\frak{f}$ by
\begin{equation}\label{ident}
(\alpha_1X,\alpha_2X,\dots,\alpha_{n+1}X)\mapsto \bigl((\alpha_1-\alpha_{n+1})X,(\alpha_2-\alpha_{n+1})X,\dots,(\alpha_{n}-\alpha_{n+1})X \bigr).
\end{equation}
It is clear that
$$
\varphi_{n+1}^{-1} \bigl((\alpha_1 X, \alpha_2 X,\dots, \alpha_{n} X,0)\bigr)=
\bigl((\alpha_1-s)X,(\alpha_2-s)X,\dots,(\alpha_{n}-s)X, -sX \bigr),
$$
where $s=(\alpha_1+\alpha_2+\cdots+\alpha_n)/(n+1)$. The map $\varphi_{n+1}$
induces the identification between the set of $\Ad(\widetilde{H})$-invariant metrics on the Ledger~--~Obata space
$\widetilde{G}/\widetilde{H}=F^{n+1}/\diag(F)$ and the set of $\Ad(H)$-invariant metrics on $G=F^n$ (it is assumed that $Z\in \frak{p}$ and
$\varphi_{n+1}(Z) \in \frak{g}$ have the same length in the induced norms).

\subsection{Parameterization of metrics} Now, we will deal preferably with $\Ad(H)$-invariant or, equivalently, with $\ad ({\mathfrak h})$-invariant
inner products on the Lie algebra $\mathfrak g$. Sometimes we will use also the above identification of such inner products with $\Ad(\widetilde{H})$-invariant
metrics on the Ledger~--~Obata space $\widetilde{G}/\widetilde{H}=F^{n+1}/\diag(F)$.
\smallskip

Let $(\cdot,\cdot)$ be an $\ad ({\mathfrak h})$-invariant inner product on $\mathfrak g$. First of all, we choose a $(\cdot,\cdot)$-orthonormal
basis of $\mathfrak g$. By Lemma \ref{strinmod}, there exists $(a_1,\dots,a_n)\in {\mathbb R}^n$ such that every $\ad ({\mathfrak h})$-invariant
irreducible submodule $\tilde{\mathfrak g}$ is of the form
$$\tilde{\mathfrak g}=\{(a_1X,a_2X,\dots,a_nX)|X\in{\mathfrak f}\}.$$ Furthermore, there are $n$ orthogonal $\ad ({\mathfrak h})$-invariant
irreducible submodules $\mathfrak g_1,\dots,\mathfrak g_n$ of $\mathfrak g$ such that $\mathfrak g=\mathfrak g_1\oplus\cdots\oplus\mathfrak g_n$ and we
can diagonalize $(\cdot,\cdot)$ and $\langle\cdot,\cdot\rangle$ simultaneously, where, for $i=1,\dots,n$,
${\mathfrak g}_i=\{(b_{i1}X,b_{i2}X,\dots,b_{in}X)|X\in{\mathfrak f}\}$.
For any $X\in\mathfrak{f}$ with $\langle X,X \rangle =1$, define a vector space $L^X$ by
\begin{equation}\label{linsp}
L^X:=\{(a_1 X, a_2 X, a_3 X, \dots, a_{n-1} X, a_n X)\,|\, a_i \in \mathbb{R}, i=1,\dots,n\}\subset \mathfrak{g},
\end{equation}
which is naturally identified with $\mathbb{R}^n$. Moreover, $\langle \cdot, \cdot \rangle$ is identified with the standard inner product on  $\mathbb{R}^n$. Clearly, $\{(b_{i1}X,b_{i2}X,\dots,b_{in}X)\}_{i=1,2,\dots,n}$ is an orthogonal basis of $L^X$, which gives
a $(\cdot,\cdot)$-orthonormal basis $\{f_1,\dots,f_n\}$ of $L^X$. For some real numbers $c_i\not=0$ and $\alpha_{ij}$, $i,j=1,\dots,n$,
\begin{equation}\label{bases}
f_i=f_i(X)=c_i(b_{i1}X,b_{i2}X,\dots,b_{in}X)=(\alpha_{i1} X, \alpha_{i2} X, \alpha_{i3} X, \dots,  \alpha_{in} X).
\end{equation}
Let $\{X_1,\dots,X_p\}$ be a $\langle \cdot, \cdot \rangle$-orthogonal basis of $\mathfrak{f}$. Then we have a $(\cdot,\cdot)$-orthonormal basis of $\mathfrak{g}$ \begin{equation}\label{baseg}
f_{ij}=f_i(X_j)=(\alpha_{i1} X_j, \alpha_{i2} X_j,  \dots,  \alpha_{in} X_j), \quad 1\leq j \leq p, \, 1 \leq i \leq n.
\end{equation}

Obviously, this basis is determined by the non-degenerate $(n\times n)$-matrix $A=(\alpha_{ij})$. For any matrix $Q\in O(n)$,
it is easy to see that the matrix $\widetilde{A}:=QA$ also determines a $(\cdot,\cdot)$-orthonormal basis in $L^X$. Without loss of generality,
we may suppose that $A$ is lower triangle with positive elements in the main diagonal, i.e. $\alpha_{ij}=0$ for $j>i$ and $\alpha_{ii}>0$.
On the other hand, we can construct an $\ad({\mathfrak h})$-invariant inner product on $\mathfrak g$ from a non-degenerate matrix
$(\alpha_{ij})$ satisfying $\alpha_{ij}=0$ for $j>i$ and $\alpha_{ii}>0$. That is, {\it there is a one-to-one correspondence between
$\ad({\mathfrak h})$-invariant inner products on $\mathfrak{g}$ and non-degenerate $(n\times n)$-matrices $(\alpha_{ij})$
satisfying $\alpha_{ij}=0$ for $j>i$ and $\alpha_{ii}>0$}. Hence, the latter set of matrices can be used for the parameterization
of $\ad({\mathfrak h})$-invariant inner products on $\mathfrak{g}$.

\subsection{The scalar curvature computation}
Now we calculate the scalar curvature of the Riemannian metric on $F^n$ generated by a matrix $A$.
For any vector $(b_1 X, b_2 X, \dots,  b_n X)\in L^X$, \begin{equation}\label{inver1}
(b_1 X, b_2 X,  \dots,  b_n X)= \sum_{i=1}^n d_i f_i(X)
\end{equation}
for some real numbers $d_1,\dots,d_n$. It follows that $$\sum_{i=1}^n \alpha_{ij} d_i=b_j, \quad j=1,\dots,n.$$
Let $A^{-1}=(\beta_{ij})$ be the inverse matrix of $A$. Then \begin{equation}\label{inver2}
d_i=\sum_{i} \delta_{ik} d_i=\sum_{i,j} \alpha_{ij}\beta_{jk}d_i=\sum_{j} b_j\beta_{jk}
\end{equation}
Using the Formula 7.29 from \cite{Bes}, the scalar curvature $S\bigl((\cdot, \cdot)\bigr)$ of the metric $(\cdot,\cdot)$ is
\begin{equation}\label{scal1}
S\bigl((\cdot, \cdot )\bigr)=
-\frac{1}{2} \sum\limits_{i,a} B(f_{ia},f_{ia})-\frac{1}{4} \sum\limits_{i,j,k,a,b,c}
\bigl([f_{ia},f_{jb}],f_{kc}\bigr)^2.
\end{equation}
Due to \eqref{baseg}, we get $$B(f_{ij},f_{ij})= \sum_{j=1}^n \alpha_{ij}^2 B(X_j,X_j)=-\sum_{j=1}^n \alpha_{ij}^2$$
since $B(X_j,X_j)=-\langle X_j,X_j \rangle =-1$. Hence,
$$
\sum\limits_{i,a} B(f_{ia},f_{ia})=-\sum\limits_{i,a} \bigl(\sum_{a=1}^n \alpha_{ia}^2\bigr)=
-\dim(\mathfrak{f})\sum\limits_{i} \bigl(\sum_{a=1}^n \alpha_{ia}^2\bigr)=
-p\sum\limits_{i=1}^p\bigl(\sum\limits_{j=1}^n \alpha_{ij}^2\bigr)=-p\sum\limits_{i,j} \alpha_{ij}^2.
$$
Denote by $C_{ij}^k$ the structural constants of the Lie algebra $\mathfrak{f}$ in the basis $\{X_1,\dots,X_p\}$, i.e. $$[X_i,X_j]=\sum_k C_{ij}^k X_k.$$ Then we have
\begin{eqnarray*}
[f_{ia},f_{jb}]&=&[(\alpha_{i1} X_a, \alpha_{i2} X_a,  \dots,  \alpha_{in} X_a),(\alpha_{j1} X_b, \alpha_{j2} X_b,  \dots,  \alpha_{jn} X_b)]\\
&=&\left(\alpha_{i1}\alpha_{j1}[X_a,X_b],\alpha_{i2}\alpha_{j2}[X_a,X_b], \dots, \alpha_{in}\alpha_{jn}[X_a,X_b]\right) \\
&=& \left(\alpha_{i1}\alpha_{j1}\sum_k C_{ab}^k X_k,\,\alpha_{i2}\alpha_{j2}\sum_k C_{ab}^k X_k, \dots, \alpha_{in}\alpha_{jn}\sum_k C_{ab}^k X_k\right).
\end{eqnarray*}
Hence, by \eqref{inver1},
\begin{eqnarray*}
\bigl([f_{ia},f_{jb}], f_{kc}\bigr)&=&
\Bigl(\Bigl(\alpha_{i1}\alpha_{j1}\sum_k C_{ab}^k X_k, \dots, \alpha_{in}\alpha_{jn}\sum_k C_{ab}^k X_k\Bigr),
\bigl(\alpha_{k1} X_c,   \dots,  \alpha_{kn} X_c\bigr)\Bigr)\\
&=&\Bigl(\bigl(\alpha_{i1}\alpha_{j1} C_{ab}^c X_c, \dots, \alpha_{in}\alpha_{jn} C_{ab}^c X_c\bigr),
\bigl(\alpha_{k1} X_c,  \dots,  \alpha_{kn} X_c\bigr)\Bigr).
\end{eqnarray*}
Similar to \eqref{inver1}, assume that $\bigl(\alpha_{i1}\alpha_{j1} C_{ab}^c X_c, \dots, \alpha_{in}\alpha_{jn} C_{ab}^c X_c\bigr)=\sum_{i=1}^n d_i f_i(X)$.
According to \eqref{inver2}, we get
$d_m=\sum_{l} b_l\beta_{lm}=\sum_{l} \alpha_{il}\alpha_{jl} \beta_{lm} C_{ab}^c.$ Therefore,
\begin{eqnarray*}
\bigl([f_{ia},f_{jb}],f_{kc}\bigr)^2=
d_k^2= \left(C_{ab}^c \right)^2 \left( \sum_{l} \alpha_{il}\alpha_{jl} \beta_{lk} \right)^2.
\end{eqnarray*}
Together with
\begin{eqnarray*}
\sum\limits_{a,b,c}\left(C_{ab}^c \right)^2&=&\sum\limits_{a,b,c} \langle [X_a,X_b],X_c\rangle^2=\sum\limits_{a,b} \langle [X_a,X_b],[X_a,X_b]\rangle\\
&=& -\sum\limits_{a,b} \langle [X_a, [X_a,X_b],X_b\rangle =-\sum\limits_{a} B(X_a,X_a)\\
&=&\dim(\mathfrak{f})=p,
\end{eqnarray*}
we have
\begin{eqnarray*}
\sum\limits_{i,j,k,a,b,c}
\bigl([f_{ia},f_{jb}],f_{kc}\bigr)^2 & =&
\sum\limits_{i,j,k,a,b,c} \left(C_{ab}^c \right)^2 \left( \sum_{l} \alpha_{il}\alpha_{jl} \beta_{lk} \right)^2 \\
&=&
\sum\limits_{a,b,c}\left(C_{ab}^c \right)^2 \cdot \sum\limits_{i,j,k} \left( \sum_{l} \alpha_{il}\alpha_{jl} \beta_{lk} \right)^2 \\
&=& p\cdot \sum\limits_{i,j,k} \left( \sum_{l} \alpha_{il}\alpha_{jl} \beta_{lk} \right)^2.
\end{eqnarray*}
Let $\Lambda_{i,j,k}=\Bigl( \sum_{l} \alpha_{il}\alpha_{jl} \beta_{lk} \Bigr)^2$. By \eqref{scal1}, we have
\begin{equation}\label{scal2}
\widetilde{S}=\frac{4}{p}\cdot S\bigl((\cdot, \cdot )\bigr)=
2\sum\limits_{i,j} \alpha_{ij}^2-\sum\limits_{i,j,k}\Lambda_{i,j,k}.
\end{equation}

\subsection{Simplification} The above formula is the key element for applying the variation principle to search invariant Einstein metrics on Ledger~--~Obata spaces.
\smallskip

First of all, $\Lambda_{i,j,k}=\Lambda_{j,i,k}$ for all indices. Then it is enough to study the case $j \geq i$.
Suppose $A$ is lower triangle. Then we have $\Lambda_{i,j,k}=0$ for $k >i$.
Moreover, for any $i$, it is clear that $$\beta_{ii}=(\alpha_{ii})^{-1}, \quad \beta_{i(i-1)}=-\alpha_{i(i-1)}\bigl(\alpha_{(i-1)(i-1)}\cdot\alpha_{ii}\bigr)^{-1},$$
which implies $\Lambda_{i,i,i}=\alpha_{ii}^2$ and $\Lambda_{i,j,i}=\Lambda_{j,i,i}=\alpha_{ji}^2$ for $j>i$. It follows that
\begin{equation}\label{scal3}
\widetilde{S}=\frac{4}{p}\cdot S\bigl((\cdot, \cdot )\bigr)=
\sum\limits_{i} \alpha_{ii}^2-\sum\limits_{k < \min(i,j)}\Lambda_{i,j,k}.
\end{equation}
Note that the volume condition means
\begin{equation}\label{vol1}
\det(A)=\prod_{i=1}^n \alpha_{ii}=1.
\end{equation}
Let us note also the following simple lemma.
\begin{lem}\label{prod1}
If the matrix $A$ is block-diagonal, i.e. $A=\diag(A_1,A_2)$, where $A_1$ is a $(k\times k)$-matrix and $A_2$ is an $(l\times l)$-matrix with $k+l=n$,
the corresponding Riemannian metric on $G=F^n$ is a product of some $\Ad(\diag(F))$-invariant metrics on $F^k$ and $F^l$.
\end{lem}

\subsection{Matrices generating isometric metrics}\label{matmet}

We will do a further discussion on metrics determined by the above matrices. Before that, we list three kinds of isometries.

\begin{pred}\label{isommat} There are the following isometries between invariant Riemannian metrics determined by different matrices.

{\rm (i)} Two non-degenerate matrices $A$ and $Q\cdot A$ {\rm(}see  \eqref{bases}{\rm)} determine the same metric on $G=F^n$, where $Q$ is an orthogonal matrix.
In particular, one can change the sign in some rows of $A$ or interchange some rows without changing the metric.

{\rm (ii)} If a matrix $B$ is obtained by interchanging two columns in a matrix $A$, then these two matrices generate isometric metrics,
and a suitable isometry generates by interchanging two copies of $F$ in $G=F^n$.

{\rm (iii)} The metrics corresponding to the matrices $A$ and $A \cdot T^k$ are isometric for any $k=1,\dots,n$, where
$T^k=(t^k_{ij})$ is an $(n\times n)$-matrix with $t^k_{kj}=-1$ for any $j=1,\dots,n$, $t^k_{ii}=1$ for any $i\neq k$ and zero in other entries.
\end{pred}

\begin{proof}
The first two assertions are clear, we will prove the third one. Recall that the group $\widetilde{G}=F^{n+1}$ acts transitively on $G=F^n$ by
$$
(a_1,a_2,\dots,a_n,a_{n+1})\circ (x_1,x_2,\dots, x_n)=(a_1x_1a_{n+1}^{-1},\,a_2x_2a_{n+1}^{-1},\,\dots,\, a_nx_na_{n+1}^{-1})
$$
with the isotropy group $\widetilde{H}=\diag(F)\subset \widetilde{G}=F^{n+1}$ at the unit of $G=F^n$,
which identifies $F^n$ with the Ledger~--~Obata space $F^{n+1}/\diag(F)$.
There are several $\Ad(\widetilde{H})$-invariant complements to
$\widetilde{\frak{h}}$ in $\widetilde{\frak{g}}$.
More precisely, by Lemma \ref{strinmod} every irreducible $\Ad(\widetilde{H})$-invariant submodule in $\widetilde{\frak{g}}$ has the form (\ref{irred}):
\begin{equation*}
\left\{(\alpha_1 X, \alpha_2 X,\dots, \alpha_{n+1} X)\subset\widetilde{\frak{g}}\,|\, X \in \frak{f}\right\}
\end{equation*}
for some fixed $\alpha_i$. In particular, we get the isotropy
subalgebra $\widetilde{\frak{h}}$ for
$\alpha_1=\alpha_2=\cdots=\alpha_{n+1}\neq 0$. Recall that
$\frak{p}$ denotes the union of all submodules \eqref{irred} with
$\alpha_1+\alpha_2+\cdots+\alpha_{n+1}=0$, i.~e. $\frak{p}$ is the
$\langle \cdot,\cdot \rangle$-orthogonal complement to
$\widetilde{\frak{h}}$ in $\widetilde{\frak{g}}$. Denote by
$\frak{p}_i$, $i=1,\dots,n+1$, the union of all submodules
\eqref{irred} with $\alpha_i=0$.
\smallskip

Each module $\frak{p}_i$ is naturally identified with $\frak{g}=n\cdot \frak{f}$. In fact, define $\psi_{n+1}:\frak{p}_{n+1} \rightarrow \frak{g}$ by
$$\psi_{n+1}(X_1,X_2,\dots,X_n,0)=(X_1,X_2,\dots,X_n)$$
and $\psi_{i}:\frak{p}_{i} \rightarrow \frak{g}$ for $i\leq n$ by
$$\psi_{i}(X_1,\dots,X_{i-1},0,X_{i+1},\dots,X_n,X_{n+1})=(X_1,\dots,X_{i-1},X_{n+1},X_{i+1},\dots,X_n).$$
This gives the identification between $\mathfrak p_i$ with $\mathfrak g$, and then induces the identification between the set of $\Ad(\widetilde{H})$-invariant metrics on
the Ledger~--~Obata space $\widetilde{G}/\widetilde{H}=F^{n+1}/\diag(F)$ and the set of $\Ad(H)$-invariant metrics on $G=F^n$.
\smallskip

Define the linear isometry $\varphi_i:\frak{p}\rightarrow \frak{p}_i$ by
$$
(\alpha_1X,\alpha_2X,\dots,\alpha_{n+1}X)\mapsto \bigl((\alpha_1-\alpha_i)X,(\alpha_2-\alpha_i)X,\dots,(\alpha_{n+1}-\alpha_i)X \bigr).
$$
It is clear that
$$
\varphi_{n+1}^{-1} \bigl((\alpha_1 X, \alpha_2 X,\dots, \alpha_{n} X,0)\bigr)=
\bigl((\alpha_1-s)X,(\alpha_2-s)X,\dots,(\alpha_{n}-s)X, -sX \bigr),
$$
where $s=(\alpha_1+\alpha_2+\cdots+\alpha_n)/(n+1)$. Hence,
\begin{eqnarray*}
&&\varphi_k\circ\varphi_{n+1}^{-1} \bigl((\alpha_1 X, \alpha_2 X,\dots, \alpha_{n} X,0)\bigr)\\
&=&\bigl((\alpha_1-\alpha_k)X,\dots,(\alpha_{k-1}-\alpha_k)X,0,(\alpha_{k+1}-\alpha_k)X, \dots,(\alpha_{n}-\alpha_k)X, -\alpha_k X \bigr).
\end{eqnarray*}
That is to say, the metrics corresponding to $A$ on $\frak{p}_{n+1}$ and $A \cdot T^k$ on $\frak{p}_{k}$ generate the
same metric on $\frak{p}$, in particular, they are isometric under $\psi_k\circ \varphi_k\circ\varphi_{n+1}^{-1}\circ \psi_{n+1}^{-1}$.
\end{proof}

\subsection{Representation of the standard metric on $F^{n+1}/\diag(F)$}\label{stanmet} Note that there is exactly one
lower triangle matrix $A$ with positive elements in the main diagonal, that represent the Riemannian metric on $F^n$ which is isometric
to the standard metric $\rho_{st}$ on the Ledger~--~Obata space $F^{n+1}/\diag(F)$.
Indeed, by (\ref{ident}) the rows $\eta_1,\dots,\eta_n \in \mathbb{R}^{n+1}$ of the matrix $\psi_{n+1}(A)$ satisfy the equality
$(\eta_i,\eta_j)=\delta_{ij}$ with respect to the canonical inner product on $\mathbb{R}^{n+1}$.
Moreover, $(\eta_i,\eta_0)=0$, where $\eta_0=(1,1,\dots,1)$. By the Gram~--~Schmidt process we easily see that there is exactly one such matrix $A=A(\rho_{st})$.
More precisely, for $i=1,\dots,n$, we have
$$
\eta_i=\frac{1}{\sqrt{(n-i+1)(n-i+2)}}\Bigl(\underbrace{0,\dots,0}_{i-1},\, n-i+1,\,\underbrace{-1,\dots,-1}_{n-i+1} \Bigr)
$$
and the entries $a_{ij}$ of the matrix $A$ are the following:
\begin{equation}\label{prodstand}
a_{ij}=\frac{1}{\sqrt{(n-i+1)(n-i+2)}}\mbox{\,  for  \,}j<i,\quad a_{ii}=\sqrt{\frac{n-i+2}{n-i+1}}\mbox{\, and \,}
a_{ij}=0 \mbox{\, for \,}j>i.
\end{equation}
Note that $\det(A)=\prod_{i=1}^n a_{ii}=\sqrt{n+1}$ and the scalar curvature of the corresponding Riemannian metric is
$S(\rho_{st})=\frac{p}{4}\cdot \widetilde{S}=\frac{p(n+3)n}{4(n+1)}$ and $\Ric(\rho_{st})=\frac{n+3}{4(n+1)} \rho_{st}$.
In particular, if we multiply the matrix $A=A(\rho_{st})$ by $\sqrt{\frac{2(n+1)}{n(n+3)}}$, then we get the metric with
$\widetilde{S}=2$, but if we multiply it by $\sqrt{\frac{4(n+1)}{(n+3)}}$ then we get the matrix $A_n^{st}$, that represent
the Einstein metric with the Einstein constant $1$, see (\ref{scal2}).
In particular,
$A_1^{st}=(2)$,
$A_2^{st}=\sqrt{\frac{6}{5}}\left(\begin{array}{cc}
\sqrt{3} & 0\\
1 & 2\\
\end{array}\right)$, and
$A_3^{st}=\frac{1}{3}\left(\begin{array}{ccc}
4\sqrt{2} & 0 & 0\\
2 & 6 & 0\\
  2\sqrt{3} & 2\sqrt{3} & 4\sqrt{3} \\
\end{array}\right)$.

\section{Einstein metrics on $F^{n+1}/\diag(F)$ for small $n$}\label{section3}

Now we describe invariant Einstein metrics on $F^2/\diag(F)$ and $F^3/\diag(F)$ in a simpler manner than that in \cite{Nik2002}.
\smallskip

If $n=1$ then $A=(\alpha_{11})$, and there is a one-parameter family of invariant (symmetric) metrics with $\widetilde{S}=\alpha_{11}^2$.
\smallskip

If $n=2$, then
$A=\left(%
\begin{array}{cc}
  \alpha_{11} & 0 \\
  \alpha_{21} & \alpha_{22} \\
\end{array}%
\right), $
and
\begin{eqnarray*}
\widetilde{S}&=&\alpha_{11}^2+\alpha_{22}^2-\Lambda_{2,2,1}=\alpha_{11}^2+\alpha_{22}^2-\left(\alpha_{21}^2\beta_{11}-\alpha_{22}^2\beta_{21}\right)^2 \\
&= & \alpha_{11}^2+\alpha_{22}^2-\left(\alpha_{21}^2\alpha_{11}^{-1}-\alpha_{21}\alpha_{22}\alpha_{11}^{-1}\right)^2.
\end{eqnarray*}
Taking $\alpha_{11}=x$, $\alpha_{22}=y$ and $\alpha_{21}=ux$, we get
$$
\widetilde{S}=x^2+y^{2}-(u^2x-uy)^2=x^2+y^{2}-{u^4}{x^2}+2{u^3}{xy}-{u^2}{y^2}.
$$
Thus we need to find all the critical points of this function under the following conditions: $$x,y>0, \quad xy=1, \quad u \in \mathbb{R}.$$
Consider the Lagrange problem: to find the critical points of the function
\begin{eqnarray}
F:=\widetilde{S}-\lambda(xy-1)&=& x^2+y^{2}-{u^4}{x^2}+2{u^3}{xy}-{u^2}{y^2}-\lambda(xy-1)\notag
\end{eqnarray}
for $x,y,u, \lambda \in \mathbb{R}$, $x,y>0$. Then we get
\begin{eqnarray}
&&\frac{\partial{F}}{\partial x} x=\frac{\partial\widetilde{S}}{\partial x}x-\lambda xy=\frac{\partial\widetilde{S}}{\partial x}x-\lambda=0,\notag\\
&&\frac{\partial{F}}{\partial y} y=\frac{\partial\widetilde{S}}{\partial x}y-\lambda xy=\frac{\partial\widetilde{S}}{\partial y}y-\lambda=0, \label{n=2}\\
&&\frac{\partial\widetilde{S}}{\partial u}=0, \quad
xy=1.\notag
\end{eqnarray}
Note that $\widetilde{S}$ is a $2$-form with respect to $x$ and $y$. In particular,
$$2\lambda=\frac{\partial\widetilde{S}}{\partial x}x+\frac{\partial\widetilde{S}}{\partial y}y=2\widetilde{S}.$$
Since the scalar curvature of a compact homogeneous Einstein non-flat manifold should be positive, we know $\lambda>0$.
Since $\widetilde{S}$ is a $2$-form with respect to $x$ and $y$, we may suppose $\lambda=2$ instead of $xy=1$, which is equivalent to
$\widetilde{S}=2$. Hence we get the following system:
\begin{eqnarray}
\frac{\partial\widetilde{S}}{\partial x}x=
\frac{\partial\widetilde{S}}{\partial y}y=2, \quad
\frac{\partial\widetilde{S}}{\partial u}=0. \label{n=2-1}
\end{eqnarray}
Any solution of \eqref{n=2-1} generates a solution of \eqref{n=2} (one only need to multiply $x,y$ by the constant $\mu$ such that $\mu^2xy=1$) and conversely.
Hence it is enough to solve the system \eqref{n=2-1} to get all Einstein metrics up to a homothety.
Note that isometric Einstein metrics with $\widetilde{S}=2$ have the same value $\widetilde{V}=xy$. By $\frac{\partial\widetilde{S}}{\partial u}=0$, we have $$u(ux-y)(2ux-y)=0.$$ It follows that $$u=0,\quad \mbox{ or } ux=y, \quad \mbox{ or } 2ux=y.$$
Furthermore, it is easy to get four solutions $(x,y,u)$:
$$(1,1,0),\quad (1,1,1), \quad \left(\frac{\sqrt{2}}{2},\sqrt{2},1\right), \quad \left(\frac{3}{\sqrt{10}},\sqrt{\frac{6}{5}},\frac{\sqrt{3}}{3}\right).$$
By Proposition~\ref{isommat} and the following facts:
$$\left(%
\begin{array}{cc}
   1 & 0 \\
   0 & -1 \\
\end{array}%
\right)\left(%
\begin{array}{cc}
    1 & 0\\
   1 & 1 \\
\end{array}%
\right)\left(%
\begin{array}{cc}
     1 & 0\\
    -1 & -1 \\
\end{array}%
\right)=\left(%
\begin{array}{cc}
    1 & 0\\
   0 & 1 \\
\end{array}%
\right),$$
$$\left(%
\begin{array}{cc}
   -\frac{\sqrt{2}}{2} & \frac{\sqrt{2}}{2}\\
   \frac{\sqrt{2}}{2} & \frac{\sqrt{2}}{2} \\
\end{array}%
\right)\left(%
\begin{array}{cc}
  \frac{\sqrt{2}}{2} & 0\\
  \frac{\sqrt{2}}{2} & \sqrt{2} \\
\end{array}%
\right)\left(%
\begin{array}{cc}
  0 & 1\\
  1 & 0 \\
\end{array}%
\right)=\left(%
\begin{array}{cc}
    1 & 0\\
    1 & 1 \\
\end{array}%
\right),$$
the first three matrices determine isometric Einstein metrics with $\widetilde{V}=1$. For the fourth one, $$\widetilde{V}=\frac{3\sqrt{3}}{5}\not=1.$$ That is, there are exactly two invariant Einstein metrics on $F^3/\diag(F)$ up to isometry and homothety (\cite{Nik2002}).

\section{Invariant Einstein metrics on $F^4/\diag(F)$}\label{3}
This section is to compute invariant Einstein metrics on $F^4/\diag(F)$. For this case, $n=3$,
$A=\left(%
\begin{array}{ccc}
  \alpha_{11} & 0 & 0\\
  \alpha_{21} & \alpha_{22} & 0\\
  \alpha_{31} & \alpha_{32} & \alpha_{33}\\
\end{array}%
\right)$,
and we get a $6$-parameter family of invariant metrics.
It is easy to get
$$A^{-1}=\left(%
\begin{array}{ccc}
  \alpha_{11}^{-1} & 0 & 0\\
  -\alpha_{21}\bigl(\alpha_{11}\cdot\alpha_{22}\bigr)^{-1} & \alpha_{22}^{-1} & 0\\
  \frac{\alpha_{32}\alpha_{21}-\alpha_{31}\alpha_{22}}{\alpha_{11}\alpha_{22}\alpha_{33}} & -\alpha_{32}\bigl(\alpha_{22}\cdot\alpha_{33}\bigr)^{-1} & \alpha_{33}^{-1}\\
\end{array}%
\right).$$
It follows that
\begin{eqnarray*}
&&\Lambda_{2,2,1}= \left(\alpha_{21}^2\alpha_{11}^{-1}-\alpha_{21}\alpha_{22}\alpha_{11}^{-1}\right)^2, \\
&&\Lambda_{2,3,1}= \Lambda_{3,2,1}=\left(\alpha_{21}\alpha_{31}\alpha_{11}^{-1}-\alpha_{32}\alpha_{21}\alpha_{11}^{-1}\right)^2,\\
&&\Lambda_{3,3,2}=\left(\alpha_{32}^2\alpha_{22}^{-1}-\alpha_{32}\alpha_{33}\alpha_{22}^{-1}\right)^2,\\
&&\Lambda_{3,3,1}=\frac{\bigl(\alpha_{31}^2\alpha_{22}-\alpha_{32}^2\alpha_{21}+\alpha_{33}(\alpha_{32}\alpha_{21}-\alpha_{31}\alpha_{22})\bigr)^2}
{\alpha_{11}^{2}\alpha_{22}^{2}}.
\end{eqnarray*}
Hence, we get
\begin{eqnarray*}
\widetilde{S}=\alpha_{11}^2+\alpha_{22}^2+\alpha_{33}^2
-\left(\alpha_{21}^2\alpha_{11}^{-1}-\alpha_{21}\alpha_{22}\alpha_{11}^{-1}\right)^2-
2\left(\alpha_{21}\alpha_{31}\alpha_{11}^{-1}-\alpha_{32}\alpha_{21}\alpha_{11}^{-1}\right)^2\\
-\left(\alpha_{32}^2\alpha_{22}^{-1}-\alpha_{32}\alpha_{33}\alpha_{22}^{-1}\right)^2
-\frac{\bigl(\alpha_{31}^2\alpha_{22}-\alpha_{32}^2\alpha_{21}+\alpha_{33}(\alpha_{32}\alpha_{21}-\alpha_{31}\alpha_{22})\bigr)^2}
{\alpha_{11}^{2}\alpha_{22}^{2}}.
\end{eqnarray*}
Let $\alpha_{11}=x$, $\alpha_{22}=y$, $\alpha_{33}=z$, $\alpha_{21}=u\cdot x$, $\alpha_{31}=v\cdot x$, and $\alpha_{32}=w\cdot y$. We get
\begin{eqnarray}\label{scal4}
\widetilde{S}=x^2+y^2+z^2
-u^2(u x\!-\!y)^2-2u^2(vx\!-\!wy)^2
-w^2(wy\!-\!z)^2
-\bigl(v^2x\!-\!w^2uy\!+\!(wu\!-\!v)z\bigr)^2.
\end{eqnarray}
The following is to find all the critical points of this function under conditions: $$x,y,z>0, \quad xyz=1, \quad u,v,w \in \mathbb{R}.$$
Consider the Lagrange problem: to find the critical points of the function
\begin{eqnarray}
F:=\widetilde{S}-\lambda(xyz-1)&=& x^2+y^2+z^2-u^2(u x\!-\!y)^2-2u^2(vx\!-\!wy)^2-w^2(wy\!-\!z)^2 \label{scal6} \\
&& -\bigl(v^2x\!-\!w^2uy\!+\!(wu\!-\!v)z\bigr)^2-\lambda(xyz-1)\notag
\end{eqnarray}
for $x,y,z, u,v,w, \lambda \in \mathbb{R}$, $x,y,z >0$. Then we get
\begin{eqnarray}
&&\frac{\partial{F}}{\partial x} x=\frac{\partial\widetilde{S}}{\partial x}x-\lambda xyz=\frac{\partial\widetilde{S}}{\partial x}x-\lambda=0,\notag\\
&&\frac{\partial{F}}{\partial y} y=\frac{\partial\widetilde{S}}{\partial x}y-\lambda xyz=\frac{\partial\widetilde{S}}{\partial y}y-\lambda=0,\notag\\
&&\frac{\partial{F}}{\partial z} z=\frac{\partial\widetilde{S}}{\partial x}z-\lambda xyz=\frac{\partial\widetilde{S}}{\partial z}z-\lambda=0,\label{scal7}\\
&&\frac{\partial\widetilde{S}}{\partial u}=0,\quad \frac{\partial\widetilde{S}}{\partial v}=0,\quad
\frac{\partial\widetilde{S}}{\partial w}=0,\quad
xyz=1.\notag
\end{eqnarray}
Note that $\widetilde{S}$ is a $2$-form with respect to $x$, $y$, $z$. In particular,
$$3\lambda=\frac{\partial\widetilde{S}}{\partial x}x+\frac{\partial\widetilde{S}}{\partial y}y+\frac{\partial\widetilde{S}}{\partial z}z=2\widetilde{S}.$$
Since the scalar curvature of a compact homogeneous Einstein non-flat manifold should be positive, we know $\lambda>0$.
Since $\widetilde{S}$ is a $2$-form with respect to $x$, $y$, $z$, we may suppose $\lambda=2$ instead of $xyz=1$, which is equivalent to
$\widetilde{S}=3$. Hence we get the following system:
\begin{eqnarray}
\frac{\partial\widetilde{S}}{\partial x}x=
\frac{\partial\widetilde{S}}{\partial y}y=
\frac{\partial\widetilde{S}}{\partial z}z=2, \quad
\frac{\partial\widetilde{S}}{\partial u}=\frac{\partial\widetilde{S}}{\partial v}=
\frac{\partial\widetilde{S}}{\partial w}=0. \label{scal8}
\end{eqnarray}
Any solution of \eqref{scal8} generates a solution of \eqref{scal7} (one only need to multiply $x,y,z$ by the constant $\mu$ such that $\mu^3xyz=1$) and conversely.
Hence it is enough to solve the system \eqref{scal8} to get all Einstein metrics up to a homothety. Note that isometrical Einstein metrics with $\widetilde{S}=3$ have the same value $\widetilde{V}=xyz$.
\smallskip

{\small
\renewcommand{\arraystretch}{1.5}
\begin{table}[p]\label{table1}
{\bf Table 1. $\Ad(H)$-invariant Einstein metrics on $F\times F\times F$}
\begin{center}
\begin{tabular}
{|p{0.03\linewidth}|p{0.065\linewidth}|p{0.065\linewidth}|p{0.065\linewidth}|p{0.06\linewidth}|p{0.06\linewidth}|p{0.06\linewidth}|p{0.08\linewidth}
|p{0.08\linewidth}|p{0.08\linewidth}|p{0.08\linewidth}|}
\hline
N&$x\!=\!a_{11}$&$y\!=\!a_{22}$&$z\!=\!a_{33}$&$u$&$v$&$w$&$a_{21}\!=\!ux$&$a_{31}\!=\!vx$ &$a_{32}\!=\!wy$ & $\widetilde{V}\!=\!xyz$\\
\hline \hline
1& 1& 1&1 &0 & 0&0 &0 &0 &0 &1  \\
\hline
2&1 & 1& 1&1 &0 & 0&1 &0 &0 & 1 \\
\hline
3&1 &1 &1 & 0&1 &0 &0 &1 &0 &1  \\
\hline
4&1 &1 & 1&0 & 0& 1&0 &0 &1 &1  \\
\hline
5&1 &1 &1 &0 &1 &1 &0 &1 &1 &1  \\
\hline
6&1 &1 &1 &1 &1 &1 &1 &1 &1 &1  \\
\hline
7&1 &$\frac{1}{\sqrt{2}}$ &$\sqrt{2}$ &0 &0 &1 &0 &0 &$\frac{1}{\sqrt{2}}$ &1  \\
\hline
8&1 &$\frac{1}{\sqrt{2}}$ &$\sqrt{2}$ &0 &$\sqrt{2}$ &1  &0 & $\sqrt{2}$& $\frac{1}{\sqrt{2}}$&1  \\
\hline
9&1 &$\frac{1}{\sqrt{2}}$ &$\sqrt{2}$ &$\frac{1}{\sqrt{2}}$ &$\frac{1}{\sqrt{2}}$ &1 &$\frac{1}{\sqrt{2}}$ &$\frac{1}{\sqrt{2}}$ &$\frac{1}{\sqrt{2}}$ &1  \\
\hline
10&$\frac{1}{\sqrt{3}}$ &$\sqrt{\frac{3}{2}}$ &$\sqrt{2}$ &$\frac{1}{\sqrt{2}}$ &$\sqrt{\frac{3}{2}}$ &$\frac{1}{\sqrt{3}}$ &$\frac{1}{\sqrt{6}}$ &
$\frac{1}{\sqrt{2}}$ &$\frac{1}{\sqrt{2}}$& 1 \\
\hline
11&$\frac{1}{\sqrt{2}}$ &1 &$\sqrt{2}$ &0 &1 &0 &0 &$\frac{1}{\sqrt{2}}$ &0 &1  \\
\hline
12&$\frac{1}{\sqrt{2}}$ &1 &$\sqrt{2}$ &0 &1 &$\sqrt{2}$ &0 &$\frac{1}{\sqrt{2}}$ &$\sqrt{2}$ &1  \\
\hline
13&$\frac{1}{\sqrt{2}}$ &$\sqrt{2}$ &1 &1 &0 &0 &$\frac{1}{\sqrt{2}}$ & 0& 0 & 1 \\
\hline
14&$\frac{1}{\sqrt{2}}$ &$\sqrt{2}$ &1 &1 &$\sqrt{2}$ &$\frac{1}{\sqrt{2}}$ & $\frac{1}{\sqrt{2}}$ &1 &1 &1 \\
\hline
15&$\frac{1}{\sqrt{2}}$ &$\sqrt{\frac{2}{3}}$ &$\sqrt{3}$  &$\frac{1}{\sqrt{3}}$  &$\frac{2\sqrt{6}}{3}$  &$\frac{1}{\sqrt{2}}$ &
$\frac{1}{\sqrt{6}}$ &$\frac{2}{\sqrt{3}}$ &$\frac{1}{\sqrt{3}}$ &1  \\
\hline
16&$\frac{1}{\sqrt{2}}$ &$\sqrt{\frac{2}{3}}$ &$\sqrt{3}$  &$\frac{1}{\sqrt{3}}$  &$\sqrt{\frac{2}{3}}$ &$\sqrt{2}$  &$\frac{1}{\sqrt{6}}$ &
$\frac{1}{\sqrt{3}}$ & $\frac{2}{\sqrt{3}}$&1  \\
\hline
17&$\frac{2\sqrt{2}}{3}$ &1 &$\frac{2}{\sqrt{3}}$ &$\frac{1}{2\sqrt{2}}$ &$\frac{\sqrt{6}}{4}$ & $\frac{1}{\sqrt{3}}$&$\frac{1}{3}$ &
$\frac{1}{\sqrt{3}}$ &$\frac{1}{\sqrt{3}}$ &$\frac{4\sqrt{6}}{9}$   \\
\hline
18&1 &$\frac{3}{\sqrt{10}}$ &$\sqrt{\frac{6}{5}}$ &0 &0 & $\frac{1}{\sqrt{3}}$&0 &0 &$\sqrt{\frac{3}{10}}$ & $\frac{3\sqrt{3}}{5}$  \\
\hline
19&1 &$\frac{3}{\sqrt{10}}$ &$\sqrt{\frac{6}{5}}$ &0 &$\sqrt{\frac{6}{5}}$ & $\frac{1}{\sqrt{3}}$&0 &$\sqrt{\frac{6}{5}}$ &
$\sqrt{\frac{3}{10}}$ & $\frac{3\sqrt{3}}{5}$ \\
\hline
20&1 &$\frac{3}{\sqrt{10}}$ &$\sqrt{\frac{6}{5}}$ &$\frac{3}{\sqrt{10}}$ &$\sqrt{\frac{3}{10}}$& $\frac{1}{\sqrt{3}}$ &$\frac{3}{\sqrt{10}}$ &
$\sqrt{\frac{3}{10}}$&$\sqrt{\frac{3}{10}}$ & $\frac{3\sqrt{3}}{5}$ \\
\hline
21&$\frac{3}{\sqrt{19}}$ &$\sqrt{\frac{19}{10}}$ &$\sqrt{\frac{6}{5}}$ &$\frac{3}{\sqrt{10}}$& $\sqrt{\frac{19}{30}}$ &$\sqrt{\frac{3}{19}}$&
$\frac{9}{\sqrt{190}}$ &$\sqrt{\frac{3}{10}}$ &$\sqrt{\frac{3}{10}}$ &
$\frac{3\sqrt{3}}{5}$ \\
\hline
22&$\frac{3}{\sqrt{19}}$ &$\sqrt{\frac{57}{55}}$ &$\sqrt{\frac{11}{5}}$ &$\sqrt{\frac{5}{33}}$ &$2\sqrt{\frac{19}{55}}$ &$\sqrt{\frac{3}{19}}$ &
$\sqrt{\frac{15}{209}}$ &$\frac{6}{\sqrt{55}}$ &$\frac{3}{\sqrt{55}}$ &$\frac{3\sqrt{3}}{5}$ \\
\hline
23&$\frac{3}{\sqrt{19}}$ &$\sqrt{\frac{57}{55}}$ &$\sqrt{\frac{11}{5}}$ &$\sqrt{\frac{5}{33}}$ &$\frac{1}{3}\sqrt{\frac{95}{11}}$ &$\frac{8}{\sqrt{57}}$ &
$\sqrt{\frac{15}{209}}$ &$\sqrt{\frac{5}{11}}$ &$\frac{8}{\sqrt{55}}$ &$\frac{3\sqrt{3}}{5}$ \\
\hline
24&$\frac{3}{\sqrt{10}}$ &1 &$\sqrt{\frac{6}{5}}$ &0 &$\frac{1}{\sqrt{3}}$ &0 &0 &$\sqrt{\frac{3}{10}}$ &0 & $\frac{3\sqrt{3}}{5}$ \\
\hline
25&$\frac{3}{\sqrt{10}}$ &1 &$\sqrt{\frac{6}{5}}$ &0 &$\frac{1}{\sqrt{3}}$ &$\sqrt{\frac{6}{5}}$ &0 &$\sqrt{\frac{3}{10}}$ &$\sqrt{\frac{6}{5}}$ &
 $\frac{3\sqrt{3}}{5}$ \\
\hline
26&$\frac{3}{\sqrt{10}}$ &$\sqrt{\frac{6}{11}}$ &$\sqrt{\frac{11}{5}}$ &$\sqrt{\frac{5}{33}}$ &$\frac{8\sqrt{22}}{33}$&$\sqrt{\frac{5}{6}}$&
$\sqrt{\frac{3}{22}}$ &$\frac{8}{\sqrt{55}}$ &$\sqrt{\frac{5}{11}}$ &$\frac{3\sqrt{3}}{5}$ \\
\hline
27&$\frac{3}{\sqrt{10}}$ &$\sqrt{\frac{6}{11}}$ &$\sqrt{\frac{11}{5}}$ &$\sqrt{\frac{5}{33}}$ &$\sqrt{\frac{2}{11}}$ &$\sqrt{\frac{6}{5}}$&
$\sqrt{\frac{3}{22}}$  &$\frac{3}{\sqrt{55}}$ &$\frac{6}{\sqrt{55}}$ & $\frac{3\sqrt{3}}{5}$ \\
\hline
28&$\frac{3}{\sqrt{10}}$ &$\sqrt{\frac{6}{5}}$ &1 &$\frac{1}{\sqrt{3}}$ &0 &0 &$\sqrt{\frac{3}{10}}$ &0 &0 & $\frac{3\sqrt{3}}{5}$ \\
\hline
29&$\frac{3}{\sqrt{10}}$ &$\sqrt{\frac{6}{5}}$ &1 &$\frac{1}{\sqrt{3}}$ &$\frac{\sqrt{10}}{3}$ &$\sqrt{\frac{5}{6}}$ &
$\sqrt{\frac{3}{10}}$ &1 &1 & $\frac{3\sqrt{3}}{5}$ \\
\hline
\end{tabular}
\end{center}
\end{table}
}

Consider the polynomial ideal $J$ generated by the polynomials
$$
\frac{\partial\widetilde{S}}{\partial x}x-2,\quad
\frac{\partial\widetilde{S}}{\partial y}y-2,\quad
\frac{\partial\widetilde{S}}{\partial z}z-2,\quad
\frac{\partial\widetilde{S}}{\partial u},\quad
\frac{\partial\widetilde{S}}{\partial v},\quad
\frac{\partial\widetilde{S}}{\partial w}.
$$
These polynomials depend on the variables $x,y,z,u,v,w$, i.e. $J \subset \mathbb{K}[x,y,z,u,v,w]$.
Using the command \texttt{EliminationIdeal} (that eliminates variables from an ideal using a Gr\"{o}bner basis computation) from Maple,
we easily get the eliminated ideals
$$
J_1:=J\cap \mathbb{K}[x],\quad J_2:=J\cap \mathbb{K}[x,y], \quad J_3:=J\cap \mathbb{K}[x,y,z].
$$
Note that $J_1$ generates by the polynomial
$$
f(x)=(x^2-1)(2x^2-1)(3x^2-1)(9x^2-8)^2(10x^2-9)(19x^2-9)(85x^4-144x^2+64).
$$
The equation $f(x)=0$ has 6 different positive roots $x_1,\dots,x_6$. Taking $x_i$ and replacing $x$ by $x_i$ in $J_2$, we get an ideal generating by a polynomial $f_i(y)$ only depending on $y$. Solve the equation $f_i(y)=0$ and find every pair $(x>0,y>0)$ which lies in the zero set of $J_2$.
We do this for $i=1,\dots,6$ and finally have 13 such pairs. Replacing $(x,y)$ by every pair in $J_3$, we easily obtain all suitable $z$.
Note that there is exactly one suitable $z>0$ for every pair. Hence we get 13 triples $(x>0, y>0,z>0)$ which lie in the zero set of $J_3$. Replacing $(x,y,z)$ by every triple in $J$, we get a polynomial system in the variables $u$, $v$, and $w$, and then we get the values of $u$, $v$, and $w$. Here we only list the results (together with the corresponding values $a_{ij}$) in Table 1.

\section{The proof of Theorem~\ref{thm1}}
Let $A_i$ denote the matrix determined by the $i$-th row of Table 1 and let $(\cdot,\cdot)_i$ be the Einstein metric generated by $A_i$. Since all metrics in Table 1 have the same scalar curvature, but there are three different values for the volume. Thus if $(\cdot,\cdot)_i$ is isometric to $(\cdot,\cdot)_j$, then $1\leq i,j\leq 16$ or $18\leq i,j\leq 29$.
\smallskip

The following is to study the isometry between $(\cdot,\cdot)_i$ and $(\cdot,\cdot)_j$ when $1\leq i,j\leq 16$ or $18\leq i,j\leq 29$. Firstly, $(\cdot,\cdot)_1$ is isometric to $(\cdot,\cdot)_2$ by Proposition~\ref{isommat} and the discussion for $n=2$ in section~\ref{section3}. Secondly, $(\cdot,\cdot)_2$ is isometric to $(\cdot,\cdot)_3$ by Proposition~\ref{isommat} and
the following fact
$$\left(%
\begin{array}{ccc}
   1 & 0 & 0\\
  0 & 0 & 1\\
  0 & 1 & 0 \\
\end{array}%
\right)A_3\left(%
\begin{array}{ccc}
   1 & 0 & 0\\
  0 & 0 & 1\\
  0 & 1 & 0 \\
\end{array}%
\right)=\left(%
\begin{array}{ccc}
   1 & 0 & 0\\
  1 & 1 & 0\\
  0 & 0 & 1 \\
\end{array}%
\right)=A_2.$$
By the same way (just interchanging the rows and columns of the matrices), we have the following pairs and triples which determine isometric metrics: $$(A_2,A_{3},A_{4}), \quad (A_7,A_{11},A_{13}), \quad (A_8,A_{12}),\quad (A_{18},A_{24},A_{28}), \quad (A_{19},A_{25}).$$
We also know that $(\cdot,\cdot)_5$ is isometric to $(\cdot,\cdot)_1$ by Proposition~\ref{isommat} and the following fact
$$\left(%
\begin{array}{ccc}
   1 & 0 & 0\\
  0 & 1 & 0\\
  0 & 0 & -1 \\
\end{array}%
\right)A_5\left(%
\begin{array}{ccc}
   1 & 0 & 0\\
  0 & 1 & 0\\
  -1 & -1 & -1 \\
\end{array}%
\right)=\left(%
\begin{array}{ccc}
   1 & 0 & 0\\
  0 & 1 & 0\\
  0 & 0 & 1 \\
\end{array}%
\right)=A_1.$$
By the same way, we have the following pairs which determine isometric metrics:
$$(A_2,A_{6}), \quad (A_{13},A_{14}), \quad (A_{18},A_{19}), \quad (A_{28},A_{29}).$$
Finally, for the other cases, we have:
\begin{enumerate}
  \item  let $O_1=\left(%
\begin{array}{ccc}
   1 & 0 & 0\\
  0 & -\frac{\sqrt{2}}{2} & \frac{\sqrt{2}}{2}\\
  0 & \frac{\sqrt{2}}{2} & \frac{\sqrt{2}}{2} \\
\end{array}%
\right)$ and $O_2=\left(%
\begin{array}{ccc}
   1 & 0 & 0\\
  0 & 0 & 1\\
  0 & 1 & 0 \\
\end{array}%
\right)$, it is easy to check that
$$O_1A_7O_2=A_4,\quad O_1A_8O_2=A_6, \quad O_1A_9O_2=A_5.$$
\item
$\left(%
\begin{array}{ccc}
    -\frac{1}{\sqrt{3}} & \sqrt{\frac{2}{3}} & 0\\
    \sqrt{\frac{2}{3}} & \frac{1}{\sqrt{3}}  & 0 \\
    0 & 0& 1 \\
\end{array}%
\right)A_{10}\left(%
\begin{array}{ccc}
   0 & 1 & 0\\
  1 & 0 & 0 \\
  0 & 0 & 1 \\
\end{array}%
\right)=A_{9}.$
\item
$\left(%
\begin{array}{ccc}
   1 & 0 & 0\\
  0 & -\frac{1}{\sqrt{3}} & \sqrt{\frac{2}{3}}\\
  0 & \sqrt{\frac{2}{3}} & \frac{1}{\sqrt{3}} \\
\end{array}%
\right)A_{15}O_2=A_{14}, \left(%
\begin{array}{ccc}
   1 & 0 & 0\\
  0 & -\sqrt{\frac{2}{3}}& \frac{1}{\sqrt{3}}\\
  0 & \frac{1}{\sqrt{3}}& \sqrt{\frac{2}{3}} \\
\end{array}%
\right)A_{16}O_2=A_{12}.$
\item
$\left(%
\begin{array}{ccc}
    1  & 0 & 0\\
    0 & -\frac{1}{2}  &  -\frac{\sqrt{3}}{2} \\
    0 & -\frac{\sqrt{3}}{2} & \frac{1}{2} \\
\end{array}%
\right)A_{20}\left(%
\begin{array}{ccc}
   1 & 0 & 0\\
  -1 & -1 & -1 \\
  0 & 0 & 1 \\
\end{array}%
\right)=\left(%
\begin{array}{ccc}
   1 & 0 & 0\\
  0 & \frac{3}{\sqrt{10}} & 0 \\
  0 & \sqrt{\frac{3}{10}} & \sqrt{\frac{6}{5}} \\
\end{array}%
\right)=A_{18}.$
\item
$\left(%
\begin{array}{ccc}
    -\frac{3}{\sqrt{19}}  & \sqrt{\frac{10}{19}} & 0\\
    \sqrt{\frac{10}{19}} & \frac{3}{\sqrt{19}}  &  0 \\
    0 & 0 & 1 \\
\end{array}%
\right)A_{21}\left(%
\begin{array}{ccc}
   0 & 1 & 0\\
  1 & 0& 0 \\
  0 & 0 & 1 \\
\end{array}%
\right)=A_{20}.$
\item
$\left(%
\begin{array}{ccc}
    -\sqrt{\frac{5}{38}}  & \sqrt{\frac{33}{38}} & 0\\
    \sqrt{\frac{33}{38}} & \sqrt{\frac{5}{38}}  &  0 \\
    0 & 0 & 1 \\
\end{array}%
\right)A_{22}\left(%
\begin{array}{ccc}
   0 & 1 & 0\\
  1 & 0& 0 \\
  0 & 0 & 1 \\
\end{array}%
\right)=\left(%
\begin{array}{ccc}
   \frac{3}{\sqrt{10}} & 0 & 0\\
  \sqrt{\frac{3}{22}} & \sqrt{\frac{6}{11}} & 0\\
   \frac{3}{\sqrt{55}} & \frac{6}{\sqrt{55}} &  \sqrt{\frac{11}{5}} \\
\end{array}%
\right)=A_{27}.$
\item
$\left(%
\begin{array}{ccc}
    -\sqrt{\frac{5}{38}}  & \sqrt{\frac{33}{38}} & 0\\
    \sqrt{\frac{33}{38}} & \sqrt{\frac{5}{38}}  &  0 \\
    0 & 0 & 1 \\
\end{array}%
\right)A_{23}\left(%
\begin{array}{ccc}
   0 & 1 & 0\\
  1 & 0& 0 \\
  0 & 0 & 1 \\
\end{array}%
\right)=\left(%
\begin{array}{ccc}
   \frac{3}{\sqrt{10}} & 0 & 0\\
  \sqrt{\frac{3}{22}} & \sqrt{\frac{6}{11}} & 0\\
   \frac{8}{\sqrt{55}} & \sqrt{\frac{5}{11}} & \sqrt{\frac{11}{5}} \\
\end{array}%
\right)=A_{26}.$
\item
$\left(%
\begin{array}{ccc}
    1  & 0 & 0\\
    0 & -\sqrt{\frac{5}{11}}  &  \sqrt{\frac{6}{11}} \\
    0 & \sqrt{\frac{6}{11}} & \sqrt{\frac{5}{11}} \\
\end{array}%
\right)A_{26}O_2=\left(%
\begin{array}{ccc}
   \frac{3}{\sqrt{10}} & 0 & 0\\
  \sqrt{\frac{3}{10}} & \sqrt{\frac{6}{5}} & 0\\
  1 & 1 & 1 \\
\end{array}%
\right)=A_{29}.$
\item
$\left(%
\begin{array}{ccc}
    1  & 0 & 0\\
    0 & -\sqrt{\frac{5}{11}}  &  -\sqrt{\frac{6}{11}} \\
    0 & -\sqrt{\frac{6}{11}} & \sqrt{\frac{5}{11}} \\
\end{array}%
\right)A_{27}\left(%
\begin{array}{ccc}
   1 & 0 & 0\\
  -1 & -1 & -1\\
  0 & 0 & 1 \\
\end{array}%
\right)=\left(%
\begin{array}{ccc}
   \frac{3}{\sqrt{10}} & 0 & 0\\
  \sqrt{\frac{3}{10}} & \sqrt{\frac{6}{5}} & 0\\
  0 & 0 & 1 \\
\end{array}%
\right)=A_{28}.$
\end{enumerate}
Summing up, we know that $(\cdot,\cdot)_i$ is isometric to $(\cdot,\cdot)_j$ if $1\leq i,j\leq 16$ or $18\leq i,j\leq 29$, which implies Theorem~\ref{thm1}.

\begin{remark}
It is easy to see that $(\cdot,\cdot)_1$ is the standard metric on $F^3$ and $(\cdot,\cdot)_{17}$ is proportional to the standard metric on $F^4/\diag(F)$.
\end{remark}

\section{On the number of invariant Einstein metrics on $F^{n+1}/\diag(F)$}
In this section we discuss the number of invariant Einstein metrics on general Ledger~--~Obata spaces $F^{n+1}/\diag(F)$. We also consider the following related problem: the number of critical points of scalar curvature functions (restricted to the metrics of fixed volume) on these spaces.
\smallskip

Firstly we produce Einstein metrics on $F^{n+1}/\diag(F)$ using standard metrics $\rho_{st}$ of $F^m/\diag(F)$ for various natural numbers $m$. The main idea is very simple: {\it The direct product of two Einstein manifolds with the Einstein constant $\lambda$
is also an Einstein manifold with the Einstein constant $\lambda$}.
Assume that $\lambda=1$ for simplicity. The discussion in subsection \ref{stanmet} shows that the matrix $A_n^{st}$ determines an invariant Einstein metric on the group $F^n$, that has the Einstein constant $1$ and is isometric to a multiple of the standard metric $\rho_{st}$ on $F^{n+1}/\diag(F)$.
\smallskip

For any (integer) partition of a positive integer $n$, that is the way of writing $n$ as a sum of positive integers,
(two sums that differ only in the order of their summands are considered as the same partition): $n=k_1+k_2+\cdots+ k_l$, define an $(n\times n)$-matrix
\begin{equation}\label{matcomp}
A=\diag \Bigl(A_{k_1}^{st},A_{k_2}^{st},\cdots, A_{k_l}^{st}\Bigr).
\end{equation}
This matrix generates an invariant Einstein metric on $F^n$ (see the previous sections), that is a direct product of standard metrics on the spaces
$F^{k_i+1}/\diag(F)$ up to a positive multiple.
\smallskip

An invariant Einstein metric on $F^{n+1}/\diag(F)$ (and on $F^n$) is called {\it routine} if it is isometric or homothetic to some metric generated by (\ref{matcomp}).
Recall that the latter are Einstein metrics with the Ricci constant~$1$ that are the direct products of some normal metrics on the spaces $F^m/\diag(F)$.
Clearly, every invariant Einstein metric on $F^{n+1}/\diag(F)$ is routine for $n=1,2,3$.
\smallskip

\begin{ques}\label{q1} Is there a non-routine invariant Einstein metrics on the Ledger~--~Obata spaces $F^{n+1}/\diag(F)$?
\end{ques}

Note, that every routine Einstein metric is normal, hence, naturally reductive and geodesic orbit in particular
(see discussion on these classes of Riemannian manifolds e.g. in \cite{Bes,NRS}). Hence, the following question is natural.

\begin{ques}\label{q2}
Is there an invariant Einstein metric on the Ledger~--~Obata space $F^{n+1}/\diag(F)$ that is not a geodesic orbit Riemannian metric?
\end{ques}

Now we give {\it the proof of theorem~\ref{thm2}}.
Since every Riemannian manifold $(F^m/\diag(F),\rho_{st})$ is irreducible, by the de~Rham theorem, we know that different partitions of $n$
($n=k_1+k_2+\cdots +k_l$)
give nonisometric Einstein metrics generated by  (\ref{matcomp}).
Hence, the number of invariant Einstein routine metrics on $F^{n+1}/\diag(F)$ is not less than $p(n)$, the number
of integer partitions of $n$, that proves the first assertion of Theorem~\ref{thm2}.
\smallskip

For details on  the number of integer partitions we refer to the book \cite{And1976}.
We have no general formula for $p(n)$ but some of its value are the following (see details in \cite{And1976}):
\begin{eqnarray*}
 && p(1)=1, \quad p(2)=2, \quad p(3)=3,\quad  p(4)=5,\quad  p(5)=7,\quad  p(6)=11, \\
 &&  p(50)=204226,\quad  p(100)=190569292,\quad  p(200)=3972999029388.
\end{eqnarray*}
We have also the Hardy~--~Ramanujan asymptotic formula:
\begin{equation}\label{hrfor}
p(n) \sim \frac {1} {4n\sqrt3} \exp\left({\pi \sqrt {\frac{2n}{3}}}\right)\quad \mbox { as } \quad n\rightarrow \infty,
\end{equation}
see details in \cite{And1976,HarRam}. There are lower bounds for sufficiently large $n$ in \cite{Er1942,Ne1951}. It is proved by A.~Mar{\'{o}}ti in \cite{Ma2003}
that there is a lower bound
\begin{equation}\label{estimeq}
p(n)> \frac{1}{13n}\exp\left({\frac{5}{2}}\sqrt{n}\right) \quad \mbox{ for any }\quad n\geq 1.
\end{equation}
This discussion completes the proof of Theorem \ref{thm2}.
\smallskip

Let $E(n)$ denote the number of invariant Einstein metrics on $F^{n+1}/\diag(F)$ and let $CPS(n)$ denote
the number of critical points of the scalar curvature function (for metrics of a fixed volume) for $F^{n+1}/\diag(F)$.
The above results show that $E(1)=1$, $E(2)=2$, $E(3)=3$, whereas $CPS(1)=1$, $CPS(2)=4$, and $CPS(3)=29$.
Because of isometries, it is clear that $CPS(n)$ increases much faster than $E(n)$.
\smallskip

Let $REM(n)$ be the number of critical points (up to a homothety) of the scalar curvature function that are routine Einstein metrics
on the Ledger~--~Obata space $F^{n+1}/\diag(F)$. Obviously we get $CPS(n)\geq REM(n)$, and $REM(n)=CPS(n)$ for $n=1,2,3$.
By (\ref{estimeq}) we obviously get
$$REM(n)\geq E(n)>\frac{1}{13n}\exp\left({\frac{5}{2}}\sqrt{n}\right).$$

For any $n$, consider the decomposition $n=k_1+k_2+\cdots +k_l$ where every integer $k_i\geq 1$, and the routine Einstein
metric generated by the matrix
$$
A=\diag \Bigl(A_{k_1}^{st},A_{k_2}^{st},\cdots, A_{k_l}^{st}\Bigr).
$$
Clearly, the number of such decompositions is $2^{n-1}$. In fact, there is a one-to-one correspondence between the above decompositions and the subsets of signs ``+'' in the string $1+1+1+\cdots+1$ with $n$ summands. Moreover different decompositions determine distinct critical points
(even the corresponding Einstein metrics could be isometric). Hence, we easily get $$REM(n) \geq 2^{n-1}.$$

In the following we will give a stronger estimate. Firstly we recall the procedure of generating left-invariant Riemannian metrics on the Lie group $F^n$ by matrices, see subsection \ref{matmet}.

\begin{lem}\label{counme}
Let $A=(a_{ij})$ and $B=(b_{ij})$ be lower triangle non-degenerate $(n\times n)$-matrices such that $a_{nn}=b_{nn}$, $b_{ij}=a_{ij}$ for $i<n$, $b_{nj}=a_{nn}-a_{nj}$ for any $j<n-1$. Then $A$ and $B$ generate isometric metrics. Moreover, $B=A$ if and only if $2a_{nj}=a_{nn}$ for any $j<n$.
\end{lem}

\begin{proof}
It is easy to see that $B=\diag(1,1,\dots, 1, -1)\cdot A \cdot T^{n}$, where
$T^n=(t^n_{ij})$ is an $(n\times n)$-matrix with $t^n_{nj}=-1$ for any $j=1,\dots,n$, $t^n_{ii}=1$ for any $i<n$ and zero in other entries.
By Proposition \ref{isommat} we get the first assertion of the lemma.
 The second assertion is obvious.
\end{proof}

If we denote the matrix $B$ in Lemma \ref{counme} by $\widehat{A}$, then we get the map $A \mapsto \widehat{A}$ such that $\widehat{\widehat{A}\,}=A$.
Now we will prove the following proposition.

\begin{pred}\label{numbcp} The inequality $CPS(n) \geq REM(n) \geq \bigl(1+\sqrt{2}\bigr)^{n-1}$ holds for any $n\geq 1$.
\end{pred}

\begin{proof}
For any $n$, consider the decomposition $n=k_1+k_2+\cdots +k_l$ where every integer $k_i\geq 1$, and the routine Einstein
metric generated by the matrix
$$
A=\diag \Bigl(A_{k_1}^{st},A_{k_2}^{st},\cdots, A_{k_l}^{st}\Bigr).
$$

Consider another decomposition $n=s_1+s_2+\cdots +s_m$, where
$s_1=k_1+k_2$, $s_2=k_3+k_4,\dots$, $s_m=k_{\,l-2}+k_{\,l-1}+k_l$ if $l$ is odd and $s_m=k_{\,l-1}+k_l$ if $l$ is even. Note that $m=[l/2]$,
where $[x]$
means the integer part of the number $x$. Hence we have the following matrix $A=\diag \Bigl(A_1, A_2,\cdots, A_m \Bigr)$, where $A_1=\diag \Bigl(A_{k_1}^{st},A_{k_2}^{st}\Bigr)$, $A_2=\diag \Bigl(A_{k_3}^{st},A_{k_4}^{st}\Bigr)$, $\dots$,
$A_m=\diag \Bigl( A_{k_{\,l-2}}^{st}, A_{k_{\,l-1}}^{st}, A_{k_{\,l}}^{st}  \Bigr)$ if $l$ is odd and $A_m=\diag \Bigl(A_{k_{\,l-1}}^{st}, A_{k_{\,l}}^{st}  \Bigr)$ if $l$ is even.
\smallskip

Every matrix $A$ determines a routine Einstein metric, and every matrix $\diag \Bigl(B_1, B_2,\cdots, B_m \Bigr)$,
where each $B_i$ is either $A_i$ or $\widehat{A_i}$, also determines a routine Einstein metric. Since $\widehat{A_i}\neq A_i$ for every $i$, we obtain
$2^m$ routine Einstein metrics generated by a decomposition $n=k_1+k_2+\cdots +k_l$. The number of decompositions with $l$ natural summands is equal to $C_{n-1}^{l-1}$. Therefore, we get
$$
REM(n)\geq \sum_{l=1}^{n-1} C_{n-1}^{l-1} 2^{\,[l/2]}\geq \sum_{l=1}^{n-1} C_{n-1}^{l-1} 2^{\,(l-1)/2}=(1+2^{\,1/2})^{n-1}=\bigl(1+\sqrt{2}\bigr)^{n-1},
$$
which ends the proposition.
\end{proof}
\smallskip

It is clear that the estimate is not exact. Then we have the following interesting question.

\begin{ques}\label{q3} What is the number $REM(n)$ of routine Einstein metrics {\rm(}up to
multiplication of the metric by a constant\,{\rm)} on the Ledger~--~Obata space $F^{n+1}/\diag(F)$?
\end{ques}
\smallskip

Theorem \ref{thm2} leads to the following natural question.

\begin{ques}\label{q4}
For a compact homogeneous space $G/H$ denote by $NEM(G/H)$ the number of Einstein $G$-invariant Riemannian metrics on $G/H$ up to isometry and homothety.
What is
$$
EMHS:=\limsup\limits_{\dim(G/H) \to \infty} \frac{\log \bigl( NEM(G/H) \bigr)}{\sqrt{\dim(G/H)}}
$$ over all compact homogeneous spaces?
\end{ques}

For the homogeneous spaces $G/H=F^{n+1}/\diag(F)$, where $F=SU(2)$, we have $\dim(G/H)=3n$,
$NEM(G/H)\geq p(n) \sim \frac {1} {4n\sqrt3} \exp\left({\pi \sqrt {\frac{2n}{3}}}\,\right)$ as $n\rightarrow \infty$ by theorem \ref{thm2}
and the Hardy~--~Ramanujan asymptotic formula (\ref{hrfor}). Hence the number EMHS above is not less than $\sqrt{2} \pi/3 \approx 1.480960979$.
\bigskip

{\bf Acknowledgements.} The work is partially supported by Grant 1452/GF4 of Ministry of Edu\-ca\-tion and Sciences of the Republic of
Kazakhstan for 2015\,-2017 and NSF of China (No.~11571182).
The second author is very obliged to
Chern Institute of Mathematics, Nankai University, Tianjin, China,
for hospitality and visiting position while the final version of this paper
have been prepared.

\end{document}